\documentclass[11pt]{scrartcl}
\usepackage{amsmath}
\usepackage{amsfonts,amssymb}
\usepackage{graphicx}
\usepackage{enumerate}
\usepackage{amsthm}
\usepackage{relsize}

\newcommand {\mA}{\mathcal A}
\newcommand {\bR}{\mathbb R}
\newcommand {\bN}{\mathbb N}
\newcommand {\bZ}{\mathbb Z}

\newcommand {\bG}{\mathbb G}

\newcommand {\bQ}{\mathbb Q}
\newcommand {\G}{\operatorname{G}}
\newcommand {\hG}{\hat{\operatorname{G}}}
\newcommand {\Id}{\operatorname{Id}}
\newcommand {\h}{\operatorname{h}}
\newcommand {\kH}{\operatorname{H}}
\newcommand {\hh}{\hat{\operatorname{h}}}

\newcommand {\zr}{\operatorname{r}}
\newcommand {\hr}{\hat{\operatorname{r}}}
\newcommand {\zc}{\operatorname{c}}
\newcommand {\hc}{\hat{\operatorname{c}}}
\newcommand {\cF}{{\cal F}}

\newcommand {\F}{\operatorname{F}}
\newcommand {\hF}{\hat{\operatorname{F}}}
\newcommand {\ggT}{\operatorname{gcd}}
\newcommand {\RE}{\operatorname{Re}}
\newcommand {\mF}{\mathcal F}
\newcommand {\mC}{\mathcal C}

\newcommand {\p}{\mathfrak{p}}

\newcommand{\sbt}{\,\begin{picture}(-1,1)(-1,-3)\circle*{3}\end{picture}\ }

\newtheorem{theorem}{Theorem}[section]
\newtheorem {lemma}[theorem]{Lemma}
\newtheorem {proposition}[theorem]{Proposition}
\newtheorem {definition}[theorem]{Definition}
\newtheorem {corollary}[theorem]{Corollary}

\begin{document}

\title{{Positive Limit-Fourier Transform of \\Farey Fractions}}

\author{Johannes Singer\footnote{Department Mathematik, Friedrich-Alexander-Universit\"at Erlangen-N\"urnberg, Cauerstra\ss e 11, D-91058 Erlangen \newline singer@math.fau.de}}

\date{}

\maketitle

\begin{abstract}
We consider the entity of modified Farey fractions via a function $\F$ defined on the direct sum $\widehat{\prod}_{\bN}\left(\bZ/2\bZ\right)$ and we prove that $-\F$ has a non negative Limit-Fourier transform up to one exceptional coefficient.
\end{abstract}

\section{Introduction}

\subsection{The Model and Main Results}

The \emph{modified Farey sequence} is defined inductively. For $k=0$ we start with the two initial fractions $\frac{0}{1}$ and $\frac{1}{1}$. In the $k$th row we copy row number $k-1$. Between two copied fractions we adjoin their mediant. In this way we obtain Table \ref{tab:tab1}. 

We should keep in mind that the usual Farey sequence (see e.\,g. \cite{Hardy60}) does not coincide with the modified Farey sequence. Further the modified Farey sequence also arises in the left branch of the Stern-Brocot tree (\cite{Brocot,Stern}).   

If we disregard the last column on the right side of our table, we can define the Farey function $\F_k$ on $\bG_k:=\left(\bZ/2\bZ \right)^k$ by mapping the elements of $\left(\bZ/2\bZ \right)^k$ with respect to the lexicographic order to the fractions of the $k$th row read from left to right.

$\bG_k$ is a compact Abelian group with character function $\bG_k \rightarrow \mathbb{S}^1$ given by $\sigma \mapsto (-1)^{\sigma \cdot \tau}:=(-1)^{\sum_{i=1}^k{{\sigma_i \tau_i}}}$, $\tau \in \bG_k$. Hence, the dual group $\bG_k^*$ is isomorphic to $\bG_k$. Therefore we have the natural Fourier transform with respect to the Haar measure on $\bG_k$. 

We define the Abelian group $\bG_{\infty}$ as the direct sum of $\bZ/2\bZ$
\begin{align*}
\bG_{\infty}:=\widehat{\prod_{\bN}}{(\bZ /2\bZ)} \leq \prod_{\bN}(\bZ/ 2\bZ), 
\end{align*}
 which is a subgroup of the Cartesian product and consists of all elements of  $\prod_{\bN}(\bZ/ 2\bZ)$ which have only finitely many nonzero entries. Furthermore we define the projection $\p_k \colon \bG_\infty \rightarrow \bG_k$ via $\tau \mapsto (\tau_i)_{i\in \{1,\ldots,k \}}$. 
 
$\bG_\infty$ is locally compact but not compact with respect to the direct sum topology, since the direct sum topology is discrete. The Haar measure on $\bG_\infty$ is given by the counting measure. Furthermore $(\bG_\infty)^*\cong \prod_{\bN}(\bZ/2\bZ)$ because $\psi\colon \prod_{\bN}(\bZ/2\bZ) \rightarrow \operatorname{Hom}(\bG_\infty, \mathbb S^1)$, $\tau \mapsto \psi(\tau)(\sbt\,):=(-1)^{\tau \cdot\, \sbt\,}$ is a group theoretic isomorphism (for details see \cite{Ra99}).  
 
 \begin{table}
\caption{Construction of the modified Farey sequence}
\centering
\renewcommand{\arraystretch}{1.4} 
\begin{tabular}{c|ccccccccccccccccc}
$k=0$ & $\frac{0}{1}$ & & & & & & & & & & & & & & & & $\frac{1}{1}$ \\
$k=1$ & $\frac{0}{1}$ & & & & & & & & $\frac{1}{2}$ & & & & & & & & $\frac{1}{1}$ \\
$k=2$ & $\frac{0}{1}$ & & & & $\frac{1}{3}$  & & & & $\frac{1}{2}$ & & & & $\frac{2}{3}$ & & & & $\frac{1}{1}$ \\
$k=3$ & $\frac{0}{1}$ & & $\frac{1}{4}$ &  & $\frac{1}{3}$  & & $\frac{2}{5}$ & & $\frac{1}{2}$ & & $\frac{3}{5}$ & & $\frac{2}{3}$ & & $\frac{3}{4}$ & & $\frac{1}{1}$ \\
$k=4$ & $\frac{0}{1}$ & $\frac{1}{5}$ & $\frac{1}{4}$ & $\frac{2}{7}$ & $\frac{1}{3}$  & $\frac{3}{8}$ & $\frac{2}{5}$ & $\frac{3}{7}$ & $\frac{1}{2}$ & $\frac{4}{7}$ & $\frac{3}{5}$ & $\frac{5}{8}$ & $\frac{2}{3}$ & $\frac{5}{7}$ & $\frac{3}{4}$ & $\frac{4}{5}$ & $\frac{1}{1}$ 
\end{tabular}
 \label{tab:tab1}
 \end{table}
\renewcommand{\arraystretch}{1}

We regard the natural extension 
\begin{align*}
 \F \colon \widehat{\prod_{\bN}}\left(\bZ/2\bZ\right) \rightarrow [0,1[
\end{align*}
of the Farey functions $\F_k$. Furthermore we have the numerator and denominator function  
\begin{align*}
 \zr,\h \colon \widehat{\prod_{\bN}}\left(\bZ/2\bZ\right) \rightarrow \bN_0
\end{align*}
defined by the above construction via $\F=\frac{\zr}{\h}$.

On $\widehat{\prod}_{\bN}\left(\bZ/2\bZ\right) $ we define the Limit-Fourier transform as a limit of the Fourier transform on $\bG_k$. We prove that the Limit-Fourier transform of $-\F$ exists in the sense
\begin{align}\label{eq:renorm}
 j\colon \bG_{\infty} \rightarrow \bR~, ~~~j(\tau):=-\lim_{k\rightarrow \infty}{2^{-k}\sum_{\sigma\in \bG_k}{\F_k(\sigma)(-1)^{\sigma \cdot \p_k(\tau)}}}.
\end{align}

Although $\bG_\infty$ is a locally compact group we do not use the Fourier transform with respect to the Haar measure since it would be necessary that $F\in L^1(\bG_{\infty})$ which is not true (see proof of Prop. \ref{prop:L1}). Therefore we use the limit construction with scaling factor $2^{-k}$. \\
Furthermore in view of Fourier analysis it would be more natural to define $j$ on $(\bG_\infty)^*\supseteq \bG_\infty$. If we do so, we can conclude from Proposition \ref{prop:abfall} that $\operatorname{supp}(j)\subseteq \bG_{\infty}$. Hence, we incur no loss of generality using the above definition of $j$. 

\bigskip

Our main result is: 
\begin{theorem}\label{theo:ferro}
The negative Farey function $-\F$ has a non negative Limit-Fourier transform up to one exception at $\tau = 0$, i.\,e.
\begin{align*}
j \geq 0 \text{~~~~on~~~} \bG_{\infty}\setminus \{0\}. 
\end{align*}
\end{theorem}

\bigskip

Furthermore we have a number theoretic significance of the modified Farey fractions. Using the modified Farey sequence we have the following interpolation result: 
 
\begin{theorem}\label{theo:inter}
 Let $t\in [0,1]$. For $\RE(s)>2$ 
 \begin{align}\label{eq:inter}
 Z(s,t):= \sum_{\sigma \in \bG_{\infty}}{\exp\left(2\pi i t(1-\F(\sigma) \right)(h(\sigma))^{-s}}
\end{align}
is summable. Especially we have
\begin{align}\label{eq:key}
 Z(s,1)=\sum_{\sigma \in \bG_{\infty}}{\exp\left(-2\pi i\F(\sigma) \right)(h(\sigma))^{-s}} =  \frac{1}{\zeta(s)}
\end{align}
and 
\begin{align}\label{eq:zustand}
 Z(s,0)=\sum_{\sigma \in \bG_{\infty}}{(h(\sigma))^{-s}} =  \frac{\zeta(s-1)}{\zeta(s)},
\end{align}
in which $\zeta$ denotes the Riemann zeta function.
\end{theorem}

The result of this paper should be interpret as purely mathematically. However it makes sense to compare it with similar models that have a physical interpretation. We would like to emphasize that the modified Farey sequence has no known significance as model in statistical physics. 

\subsection{Physical Background and Related Models}

In this subsection we briefly review the concept of classical spin chains (\cite{Rue99}) and discuss the \emph{Number-Theoretical Spin Chain} (NTSC) of Knauf \cite{Contucci97,Contucci99,Guerra,Knauf93,Knauf98,Knauf13} and the \emph{Farey Fraction Spin Chain} (FFSC) introduced by Kleban and \"Ozl\"uk \cite{Bandtlow10,Boca07,Fiala03,Fiala04,Fiala05,Kallies01,Kesseb12,Kleban99,Peter01,Prellberg06}. 

Let $\Omega\neq \emptyset$ be a finite set. Each elementary event $\omega \in \Omega$ is assigned with an energy value $H(\omega)$. This defines an energy function $\kH\colon \Omega \to \bR$. For the inverse temperature $\beta$ the probability measure given by the density
\begin{align*}
 p_\beta \colon \Omega \to [0,1], ~~~p_\beta(\omega):=\frac{\exp(-\beta \kH(\omega))}{Z(\beta)}
\end{align*}
with \emph{partition function} $Z(\beta):=\sum_{\omega\in \Omega}\exp(-\beta H(\omega))$ is called \emph{Gibbs-measure}. For a \emph{finite spin chain} the configuration space is given by $\Omega:=E^k$ with $E:=\{\uparrow, \downarrow \} \cong \bZ/2\bZ$ and $k\in \bN$. Using the Fourier transform 
\begin{align*}
 (\cF_k f)(\tau):=2^{-k}\sum_{\sigma \in \bG_k} (-1)^{\sigma \cdot \tau}f(\sigma)~~~(\tau \in \bG_k)
\end{align*}
the energy function has the form 
\begin{align*}
 \kH_k(\sigma) = - \sum_{\tau \in \bG_k}(-1)^{\sigma \cdot \tau}j_k(\tau)
\end{align*}
with the so called \emph{interaction coefficients}
\begin{align*}
 j_k(\tau):=-(\cF_k \kH_k)(\tau) ~~~~(\tau \in \bG_k). 
\end{align*}
If $j_k \geq 0$ on $\bG_k\setminus\{0\}$ we call the spin chain weakly ferromagnetic. In this context we have to observe that for finite spin chains the interaction coefficient at $\tau = 0$ has no influence on the Gibbs measure. Therefore the weak ferromagnetism is no restriction to the general ferromagnetic case. 

For an \emph{half-infinite spin chain} we take the limit case $k\to \infty$ and obtain the configuration space $\Omega:=\bG_\infty$. The construction of the Limit-Fourier transform \eqref{eq:renorm} coincides with the thermodynamic limit in the sense of statistical mechanics. From this point of view \eqref {eq:renorm} is a natural choice. 

The class of ferromagnetic spin chains is of great importance in the context of statistical physics. First of all we have the machinery of correlation inequalities like GKS-inequalities that substantially rely on the ferromagnetic property (see \cite{Ginibre}) and are also of mathematical interest. Furthermore there is the Lee-Yang theory concerning the Ising model which predicts zeros of the partition function of a ferromagnetic spin chain (see \cite{Rue99}). \\
In a more general situation Newman studied the interplay of ferromagnetic spin chains, Lee-Yang theorem and number theory (see e.\,g. \cite{Newman74,Newman75a,Newman75b}). This emphasizes the relevance of ferromagnetism in the mathematical context.

\paragraph{Number-Theoretical Spin Chain}

The NTSC was defined by Knauf in \cite{Knauf93} via the energy function $\kH:=\log(\h)\colon \bG_\infty \to \bR$. The energy values are exactly the logarithms of the denominators of the modified Farey sequence. The partition function of this model is given by 
\begin{align*}
 Z_{\text{NTSC}}(s)=\sum_{\sigma \in \bG_{\infty}}{(h(\sigma))^{-s}} = \frac{\zeta(s-1)}{\zeta(s)}. 
\end{align*}
There are two rigorous proofs (\cite{Guerra,Knauf93}) that show that the NTSC is weakly ferromagnetic. Furthermore the partition function of Knauf's model is the starting point of the interpolation in Theorem \ref{theo:inter}. 

\paragraph{Farey Fraction Spin Chain}

In \cite{Kleban99} Kleban and \"Ozluk introduced the FFSC. The energy function is defined inductively. For $k\geq 0$ we regard $M_k\colon \bG_k \to \operatorname{SL}(2,\bZ)$ given by $M_0:=\left(\begin{smallmatrix} 1 & 0 \\ 0 & 1 \end{smallmatrix}\right)$ and for $k\geq 1$ by
\begin{align*}
 M_k(\sigma):=A^{1-\sigma_k}B^{\sigma_k}M_{k-1}(\sigma_1,\dots, \sigma_{k-1}) ~~~~(\sigma\in \bG_k)
\end{align*}
with $A:=\left(\begin{smallmatrix} 1 & 0 \\ 1 & 1 \end{smallmatrix}\right)$ and $B:=A^t= \left(\begin{smallmatrix} 1 & 1 \\ 0 & 1 \end{smallmatrix}\right)$. The energy function of the FFSC is then given by 
\begin{align*}
 \kH_k\colon \bG_k \to \bR, ~~~\kH_k(\sigma):= \log(\operatorname{Trace}(M_k)).
\end{align*}
For instance, in the case $k=2$ we get 
\begin{align*}
 M_2(0,0)=\left(\begin{smallmatrix} 1 & 0 \\ 2 & 1 \end{smallmatrix}\right),  M_2(0,1)=\left(\begin{smallmatrix} 1 & 1 \\ 1 & 2 \end{smallmatrix}\right), M_2(1,0)=\left(\begin{smallmatrix} 2 & 1 \\ 1 & 1 \end{smallmatrix}\right), 
 M_2(1,1)=\left(\begin{smallmatrix} 1 & 2 \\ 0 & 1 \end{smallmatrix}\right).
\end{align*}
We see that the Stern-Brocot tree is obtained by regarding the columns of the matrices. But in the FFSC the energy function the logarithm of the trace of theses matrices. In particular the energy  function is an extensive quantity. Therefore the model is different form our model studied in this paper. Numerical experiments (\cite{Kleban99}) indicate that the FFSC is weakly ferromagnetic. However a rigorous proof is still open (\cite{Contucci99}). 

\bigskip

Now we can interpret Theorem \ref{theo:ferro} in a similar thermodynamic spirit. The Farey function $\F$ could be considered as an energy function of an infinite spin chain even though it is not an extensive quantity. Then the Limit-Fourier transform of the negative energy function is the interaction of the spin chain and Theorem \ref{theo:ferro} tells us that our system is weakly ferromagnetic. In Theorem \ref{theo:inter} the non-extensive Farey function is inserted in the partition function of the NTSC as a phase factor which is controlled by the parameter $t$. Equation \eqref{eq:key} shows that the pole at $s=2$ vanishes in the case $t=1$.

\bigskip

This paper is organized as follows. In Section 2 we provide two equivalent definitions of the Farey function $\F_k$ which essentially rely on a set theoretic bijection between the groups $(\bZ/2\bZ)^k$ and $\bZ/2^k\bZ$.
Section 3 is devoted to properties of the Farey function which will be needed in Section 4 to show that the Limit-Fourier transform on $\bG_\infty$ is well defined and to get estimates of the interaction coefficients. Section 4 culminates in the proof of Theorem \ref{theo:ferro}. We finish the section with the proof of Theorem \ref{theo:inter}.

\section{General Framework}

Now we provide two group theoretic descriptions of the modified Farey sequence that rely  substantially on the groups
\begin{align*}
  \bG_k := (\bZ / 2 \bZ )^k ~~~ \text{and} ~~~ \G_k := \bZ /( 2^k\bZ ) 
\end{align*}
for $k\in \bN_0$. We use the complete residue systems $\{0, \ldots, 2^k-1 \}$ for $\G_k$  and $\{0,1\}$ for $\bZ / 2\bZ$. 
We represent $s\in \G_k$ uniquely in the form $s= \sum_{i=1}^k{\sigma_i 2^{k-i}}, \sigma_i\in \bZ /2 \bZ$. Therefore we get the bijection
\begin{align*} 
\Id_k \colon \G_k\rightarrow \bG_k, ~~~s\mapsto (\sigma_1, \ldots, \sigma_k). 
\end{align*}

Furthermore we define the set $\hG_k := \{0,\ldots, 2^k\}$ and the projection
\begin{align*}
\Pi_k \colon \G_k \rightarrow \hG_k,
\end{align*}
which maps every element of the cyclic group $\G_k$ to the unique representative in $\hG_k\setminus \{2^k\}$. \footnote{It would be also possible to use the group $\G_{k+1}$ instead of the set $\hG_k$. In order to avoid confusion with different indices we decided to use the set $\hG_k$.}

\cite{Knauf93} introduced for $k\in \bN$ and $s_0,s_1\in \bR$ the family  $\zr_k(s_0,s_1)\colon\bG_k\rightarrow \bR$ of functions by setting 
\begin{align*}
\zr_1(s_0,s_1)(0):=s_0,~~~~ \zr_1(s_0,s_1)(1):=s_1 
\end{align*}
and for $\sigma:=(\sigma_1, \ldots,\sigma_k) \in \bG_k$
\begin{align*}
\zr_{k+1}(s_0,s_1)(\sigma, \sigma_{k+1}):= \zr_k(s_0,s_1)(\sigma)+\sigma_{k+1}\zr_k(s_0,s_1)(1-\sigma), 
\end{align*}
in which $1-\sigma:=(1-\sigma_1, \ldots, 1-\sigma_k)$.

Now we define the function $\zc_k$ on $\bG_k$. We will see that through the right choice of the parameters $s_0$ and $s_1$ we get the numerator and denominator function.  

For $k\in \bN_0$ and $s_0,s_1\in \bR$ we define the function $\zc_k:\bG_k \rightarrow \bR$ by
\begin{align*}
\zc_k(\sigma):= \zr_{k+1}(s_0,s_1)(0,\sigma),~~~~ (\sigma \in \bG_k)
\end{align*}
and $\hc_{k}\colon\hG_k\rightarrow \bR $ inductively by $\hc_0(0):=s_0$, $\hc_0(1):=s_1$ and for $s \in \hG_k$ by 
\begin{align*}
 \hc_{k+1}(2s):=\hc_k(s)
\end{align*}
and for $s\in \hG_k\setminus \{2^k \}$ by 
\begin{align*}
\hc_{k+1}(2s+1):=\hc_k(s)+\hc_k(s+1)
\end{align*}
$\zc_k$ and $\hc_k$ are depending on the parameter $s_0$ and $s_1$. For reasons of clarity we omit the parameters in the notation.
The function $\hc_k$ reproduces the inductive construction of the modified Farey sequence in the introduction. The following lemma shows the equivalence of this definition with that of $\zc_k$ on $\bG_k$.  

\begin{lemma}\label{lem:rid}
For $k\in \bN_0$ we have
\begin{align*}
\hc_k\circ \Pi_k \equiv \zc_k\circ \Id_k. 
\end{align*}
\end{lemma}

\begin{proof} 
For $k=0$ we have $\hc_0(0)=s_0=\zc_0$. Let $s\in \G_{k+1}$. For $s=2a$, $a\in \G_k$,
 setting $\sigma :=\Id_k(a) \in \bG_k$ and $\hat{a}:=\Pi_k(a)$. We get $\Id_{k+1}(s)=(\sigma, 0)$ and $\Pi_{k+1}(s)=2\hat{a}$. Therefore we have
\begin{align*}
\hc_{k+1}(\Pi_{k+1}(s)) & = \hc_k(\hat{a})=\hc_k(\Pi_k(a)) \\
						& = \zc_k(\Id_{k}(a)) = \zr_{k+2}(s_0,s_1)(0,\sigma, 0) \\
            & = \zc_{k+1}(\sigma,0)=\zc_{k+1}(\Id_{k+1}(s)). 
\end{align*}
For $s=2a+1$, $a\in \G_k$, setting $\hat{a}:=\Pi_k(a)$. We get $\Pi_{k+1}(s)=2 \hat{a}+1$. For $a$ with $M:=\{i\in \{1,\ldots,k\}: \left(\Id_k(a)\right)_i=0  \}\neq \emptyset$, we define $l:=\max M$. Using the notation $1_{m}:=(1,\ldots, 1)\in \bG_m$ and $0_{m}:=(0,\ldots, 0)\in \bG_m$, we get $\Id_k(a)=(\sigma_1, \ldots, \sigma_{l-1},0,1_{k-l})$ and $\Id_k(a+1)=(\sigma_1,\ldots, \sigma_{l-1}, 1,0_{k-l})$ as well as $\Id_{k+1}(s)=(\sigma_1,\ldots, \sigma_{l-1},0,1_{k-l+1})$. 
Therefore we get
\begin{align*}
\hc_{k+1}(\Pi_{k+1}(s))  = & \, \hc_k(\hat{a}) + \hc_k(\hat{a}+1) = \, \hc_k(\Pi_k(a)) +\hc_k(\Pi_k(a+1))  \\
             = &  \zc_{k}(\Id_k(a))+  \zc_{k}(\Id_k(a+1))\\
             = & \zr_{k+1}(s_0,s_1)(0,\sigma_1, \ldots, \sigma_{l-1},0,1_{k-l}) \\
	     &  +  \zr_{k+1}(s_0,s_1)(0,\sigma_1,\ldots, \sigma_{l-1}, 1,0_{k-l})\\
             = & \zr_{k+1}(s_0,s_1)(0,\sigma_1, \ldots, \sigma_{l-1},0,1_{k-l}) \\
	      & +  \zr_{k+1}(s_0,s_1)(1,1-\sigma_1,\ldots,1- \sigma_{l-1}, 1,0_{k-l})\\
             = &  \zr_{k+2}(s_0,s_1)(0,\sigma_1,\ldots, \sigma_{l-1}, 0, 1_{k-l+1}) \\
             = & \zc_{k+1}(\Id_{k+1}(s)), 
\end{align*}
in which we used
\begin{align*}
  &  \zr_{k+1}(s_0,s_1)(0,\sigma_1,\ldots,\sigma_{l-1}, 1,0_{k-l}) \\
= &  \zr_{l}(s_0,s_1)(0,\sigma_1,\ldots,\sigma_{l-1}) + \zr_{l}(s_0,s_1)(1,1-\sigma_1,\ldots,1- \sigma_{l-1})\\
= & \zr_{k+1}(s_0,s_1)(1,1-\sigma_1,\ldots,1- \sigma_{l-1}, 1, 0_{k-l}).
\end{align*}
For $a=2^k-1$ (this is exactly the case when $M=\emptyset$) we have $\Pi_{k+1}(s)=2^{k+1}-1$, $\Id_{k}(a)=1_{k}$ and $\Id_{k+1}(s)=1_{k+1}$. We get
\begin{align*}
 \hc_{k+1}(\Pi_{k+1}(s))  &  =  \hc_k(2^{k}-1) + \hc_k(2^k) \\
			  & = \hc_k(\Pi_k(a)) + \hc_0(1) =  \zc_k(\Id_k(a)) + s_1 \\
			  & = \zr_{k+1}(s_0,s_1)(0,1_k) + \zr_{k+1}(s_0,s_1)(1,0_k)\\
			  & = \zr_{k+2}(s_0,s_1)(0,1_{k+1}) \\
			  & = \zc_{k+1}(\Id_{k+1}(s)). 
\end{align*}
\end{proof}

Now we use the family $\zr_k(s_0,s_1)$ to get the numerator and denominator function. 
For $k\in \bN_0$ the denominator function $\h_k\colon\bG_k \rightarrow \bN$ is defined by  
\begin{align*}
\h_k(\sigma):= \zr_{k+1}(1,1)(0,\sigma), ~~~(\sigma \in \bG_k); 
\end{align*} 
and the numerator function $\zr_k\colon\bG_k \rightarrow \bN_0$ is defined by 
\begin{align*}
\zr_k(\sigma):= \zr_{k+1}(0,1)(0,\sigma), ~~~(\sigma \in \bG_k). 
\end{align*} 

Furthermore we extend the functions $\h_k\circ \Id_k$ and $\zr_k\circ \Id_k$ on $\hG_k$ and get for $k\in \bN_0$ the function $\hh_k\colon\hG_k\rightarrow \bN$ defined by $\hh_0(0):=1$, $\hh_0(1):=1$ and for $s\in\hG_k$ by
\begin{align*}
\hh_{k+1}(2s):=\hh_k(s)
\end{align*}
as well as for $s\in \hG_k \setminus\{2^k\}$ by
\begin{align*}
\hh_{k+1}(2s+1):=\hh_k(s)+\hh_k(s+1).
\end{align*}
$\hr_{k}\colon\hG_k\rightarrow \bN_0 $ is defined by $\hr_0(0):=0$, $\hr_0(1):=1$ and for $s\in \hG_k$ by
\begin{align*}
\hr_{k+1}(2s):=\hr_k(s)
\end{align*}
as well as for $s\in \hG_k\setminus \{2^k \}$ by
\begin{align*}
\hr_{k+1}(2s+1):=\hr_k(s)+\hr_k(s+1).
\end{align*}

Immediately, we can deduce from Lemma \ref{lem:rid}: 
\begin{corollary}\label{coro:hk}\label{coro:rk}
 For $k\in \bN_0$ we have
\begin{align*}
  \hh_k\circ \Pi_k \equiv \h_k\circ \Id_k ~~~\text{and}~~~
  \hr_k\circ \Pi_k \equiv \zr_k\circ \Id_k. 
\end{align*}
\end{corollary}

\section{Farey Function}

In this section we give a formal definition of the Farey function.  
\begin{definition}\label{deffarey}
For $k\in \bN_0$ the Farey function $\F_k\colon \bG_k \rightarrow \bQ \cap [0,1[\,$ is defined by
\begin{align*}
\F_k(\sigma):=\frac{\zr_k(\sigma)}{\h_k(\sigma)}, ~~~(\sigma \in \bG_k);
\end{align*}
and the extended Farey function $\hF_k\colon \hG_k \rightarrow \bQ \cap [0,1]$ is defined by 
\begin{align*}
\hF_k(s):=\frac{\hr_k(s)}{\hh_k(s)}, ~~~(s \in \hG_k). 
\end{align*}
\end{definition}

Due to Corollary \ref{coro:hk} we have $\hF_k\circ \Pi_k \equiv \F_k\circ \Id_k$ for $k\in \bN_0$. Therefore the heuristic definition in the introduction and the group theoretic approach coincide.

Hereafter we summarize some well known properties of the Farey function $\F_k$. For $k\in \bN_0$ we have
\begin{align}\label{mono}
0 = \hF_k(0) < \hF_k(1) < \ldots < \hF_k(2^k-1) < \hF_k(2^k)=1. 
\end{align}
Therefore with respect to the lexicographic order on $\bG_k$ the map $\sigma \mapsto \F_k(\sigma)$ is strictly increasing.
Furthermore two successive extended Farey fractions satisfy the unimodular relation, i.\,e.
for $k\in \bN_0$ and $s\in \hG_k\setminus \{2^k\}$ we have 
\begin{align}\label{ggT1}
\hh_k(s) \cdot \hr_k(s+1)- \hh_k(s+1) \cdot \hr_k(s)=1. 
\end{align}
As a consequence $\ggT\left(\hr_k(s),\hh_k(s)\right)=1$ for $k\in \bN_0$ and $s\in \hG_k$.
For $k\in \bN_0$ we define the arithmetical function $\varphi_k:\bN\rightarrow \bN_0$ by 
\begin{align*}
\varphi_k(n)  := \# \{ s\in \hG_k \setminus \{ 2^k \}: \hh_k(s)=n \} 
	       = \# \{ \sigma\in \bG_k : \h_k(\sigma)=n \}. 
\end{align*}

As shown in \cite{Knauf93}, Prop. 2.2] the function $\varphi_k$ is related to Euler's totient function $\varphi$ since for $k\in \bN_0$ we have 
$\varphi_k\leq \varphi_{k+1}\leq \varphi$ and $\varphi_k(p)=\varphi(p)$ for $p\in \{1,\ldots, k+1 \}$. All in all we obtain the following result: 

\begin{proposition}
The map $ \F \colon \bG_{\infty}\rightarrow \bQ \cap [0,1[$  is bijective.  
\end{proposition}

\section{Positivity of Limit-Fourier Transform}

\subsection{Fourier Transform}
Since $\bG_k$ is a locally compact Abelian group we have a Fourier transform with respect to the Haar measure. For $k \in \bN$ the set  
$\mA_k := \{f\colon\bG_k \rightarrow \bR \}$ of real-valued observables forms an algebra (with addition and multiplication).

The Fourier transform $\mathcal{F}_k\colon\mA_k \rightarrow \mA_k$ 
is defined by 
\begin{align*}
\left( \mathcal{F}_k f\right)\left(\tau \right):=\hat{f}(\tau):= 2^{-k}\sum_{\sigma \in \bG_k}{ (-1)^{\sigma \cdot \tau}f(\sigma)},~~~~(\tau\in \bG_k). 
\end{align*}

The interaction coefficients $j_k$ of the Farey function $\F_k$ are defined by the negative Fourier transform of $\F_k$, i.\,e.
\begin{align*}
j_k(\tau):=-\left( \mathcal{F}_k\F_k \right) (\tau)=- 2^{-k}\sum_{\sigma \in \bG_k}{(-1)^{\sigma\cdot \tau} \F_k(\sigma)} 
\end{align*}
for $\tau \in \bG_k$.

An observable $f\in \mA_k$ is called \emph{strictly ferromagnetic}, if $(\mF_k f)(\tau)\geq 0$ for all $\tau \in \bG_k$  and \emph{weakly ferromagnetic}, if $(\mF_k f)(\tau)\geq 0 $ for all $\tau \in \bG_k\setminus \{0\}$. 
Obviously, the strictly ferromagnetic observables, denoted by $\mC_k \subseteq \mA_k$, form a multiplicative cone, i.\,e. for $f,g \in \mC_k$ and $\lambda \geq 0$ we have $\lambda f \in \mC_k,f+g \in \mC_k$, and $f\cdot g \in \mC_k$. 
Now we have a necessary and sufficient condition for preserving strict ferromagnetism  under composition:  
\begin{proposition}[\cite{Knauf93}, Prop. 3.2]\label{pro:perma}
 For $y \in \, ]0,\infty]$ let $g\in C^{\infty}(]-y,y[\,,\bR)$ be a function whose expansion 
\begin{align*}
 g(x)=\sum_{i=0}^{\infty}{c_i x^i}
\end{align*}
is absolutely convergent on the interval $\,]-y,y[$. Consider the set
\begin{align*}
 Y_k:=\{f\in \mA_k: f(\bG_k)\subseteq \, ]-y,y[ \, \}
\end{align*}
of observables. Then the map $\mathcal{G}_k\colon Y_k \rightarrow \mA_k$, $f\mapsto g\circ f$ 
preserves strict ferromagnetism, i.\,e. 
\begin{align*}
 \mathcal{G}_k(\mC_k \cap Y_k)\subseteq \mC_k \text{~~~for all~} k \in \bN_0,
\end{align*}
if and only if $c_i \geq 0$ for all $i \in \bN_0$. 
\end{proposition}

\subsection{Estimates of the Interaction Coefficients}
First of all we start with a symmetry property of the Farey function. 
\begin{lemma}\label{symhr}
For $k\in \bN_0$ and $s\in \hG_k$ we have
\begin{enumerate}[(i)]
 \item ~~~$\hr_k(s)+\hr_k(2^k-s)=\hh_k(s)$;
 \item ~~~$\hh_k(2^k-s)=\hh_k(s)$.
\end{enumerate}
\end{lemma}

\begin{proof} 
It could be easily shown by induction. 
\end{proof}

Now we have a lower bound for the interaction coefficients: 
\begin{proposition}\label{prop:L1}
For $k\in \bN$ we have 
\begin{align*}
j_k(0)=-\frac{1}{2}\left(1 - 2^{-k}  \right) < 0
\end{align*}
and
\begin{align*}
j_k(0) \leq j_k(\tau)
\end{align*}
for all $\tau \in \bG_k$.
\end{proposition} 

\begin{proof} 
Using Lemma \ref{symhr} we have $\hF_k(2^k-s)+\hF_k(s) = 1$
for all $s\in \hG_k$, $k\in \bN$. Therefore we get 
\begin{align*}
j_k(0) & =- 2^{-k}\sum_{\sigma \in \bG_k}{\F_k(\sigma)} = - 2^{-k}\sum_{a\in\{0,\ldots,2^k-1 \}}{\hF_k(a)}\\
       & =- 2^{-k}\left(\hF_k(2^{k-1}) + \sum_{a\in\{ 1,\ldots, 2^{k-1}-1\}}{\left[\hF_k(a)+\hF_k(2^k-a) \right]} \right) \\
       & =- 2^{-1}\left(1-2^{-k} \right). 
\end{align*}
The second claim follows from \eqref{mono}. 
\end{proof}

Now we get an upper bound for the interaction coefficients. 

\begin{proposition}
For $k\in \bN$ and $\tau_0:=(1,0_{k-1})\in \bG_k$ we have
\begin{align*}
j_k(\tau_0) > 0 
\end{align*}
and
\begin{align*}
j_k(\tau_0)\geq j_k(\tau)
\end{align*}
for all $\tau \in \bG_k$. 
\end{proposition}

\begin{proof} 
We can deduce from the orthogonal relation of the characters that
\begin{align*}
j_k(\tau_0)= - \sum_{\sigma \in \bG_{k-1}}{\left(\F_k(0,\sigma)-\F_k(1,\sigma)\right)} \geq -\sum_{\sigma \in \bG_k}{(-1)^{\sigma \cdot \tau}}\F_k(\sigma) = j_k(\tau)
\end{align*}
for all $\tau\in \bG_k$. $\F_k$ is strictly increasing and non-negative, therefore we get $j_k(\tau_0)>0$.
\end{proof}

\begin{proposition}
 The thermodynamic limit of the interaction coefficients 
 \begin{align*}
  j(\tau):=\lim_{k\rightarrow \infty}{j_k(\p_k(\tau))} 
 \end{align*}
exists for all $\tau \in \bG_\infty$.
\end{proposition}

\begin{proof} We show that 
\begin{align*}
\left|j_k(\tau)-j_{k+1}(\tau,0) \right| \leq 2^{-k-1} 
\end{align*}
for $k\in \bN$ and $\tau \in \bG_k$. 
Since the series $\sum_{k\in \bN_0}{2^{-k}}$ is convergent, the Cauchy criterion implies convergence of the interaction coefficients. 
 We have 
\begin{align*}
      & j_k(\tau)-j_{k+1}(\tau,0) \\
        = &  -2^{-k}\sum_{\sigma \in \bG_k}{(-1)^{\sigma \cdot \tau}\F_k(\sigma)} +2^{-k-1}\sum_{\sigma \in \bG_{k}}{(-1)^{\sigma \cdot \tau}\left[\F_{k+1}(\sigma,0)+\F_{k+1}(\sigma,1)\right]}\\
    = &\, 2^{-k-1} \sum_{\sigma \in \bG_{k}}{(-1)^{\sigma \cdot \tau}\left[\F_{k+1}(\sigma,1)-\F_{k+1}(\sigma,0)\right]}.
\end{align*}
Therefore we get a telescoping sum
\begin{align*}
	   & \left| j_k(\tau)-j_{k+1}(\tau,0) \right| \\
      \leq &\, 2^{-k-1} \sum_{a \in \{0,\cdots, 2^k-1 \}}{\left[\hF_{k+1}(2a+1)-\hF_{k+1}(2a)\right]} \\
      \leq &\, 2^{-k-1} \sum_{a \in \{0,\cdots, 2^k-1 \}}{\left[\hF_{k+1}(2a+2)-\hF_{k+1}(2a)\right]} \\
    \leq &\, 2^{-k-1}.
\end{align*}
\end{proof}

Now we provide an upper bound of the interaction $j(\tau)$ depending on the support of $\tau \in \bG_{\infty}$.

\begin{lemma}\label{lem:sum}
 For $k\in \bN$ we have
 \begin{align*}
  \sum_{s\in \{0,\ldots,2^k-1\} }{\frac{1}{\hh_k(s)\hh_k(s+1)}}= 1.
 \end{align*}
\end{lemma}

\begin{proof}
For $k=1$ the claimed identity holds. 
By induction hypothesis we get
\begin{align*}
   ~ & \sum_{s\in \{0,\ldots,2^{k+1}-1\} }{\frac{1}{\hh_{k+1}(s)\hh_{k+1}(s+1)}} \\
   = & \sum_{a\in \{0,\ldots,2^{k}-1\}}{\left[ \frac{1}{\hh_{k+1}(2a)\hh_{k+1}(2a+1)}+\frac{1}{\hh_{k+1}(2a+1)\hh_{k+1}(2a+2)}\right]} \\
   = & \sum_{a\in \{0,\ldots,2^{k}-1\}}\frac{1}{\hh_{k}(a)\hh_{k}(a+1)}
   = 1. 
\end{align*} 
\end{proof}

\begin{proposition}\label{prop:abfall}
 For all $\tau \in \bG_{\infty}\setminus \{0 \}$ we have
 \begin{align*}
  j(\tau) \leq 2^{-\max(\operatorname{supp}(\tau))}. 
 \end{align*}
\end{proposition}

\begin{proof}
 It suffices to prove that, for $k\in \bN$ and $n\in \bN_0$ with $n \leq k-1$,  
 \begin{align*}
  j_k(\tau,1,0_{k-n-1}) \leq 2 ^{-n-1}
 \end{align*}
is valid for all $\tau \in \bG_n$. We have 
 \begin{align*}
  ~ & ~\left| j_k(\tau,1,0_{k-n-1}) \right| \\
  = & ~\left| 2^{-k}\sum_{\sigma' \in \bG_n, \sigma'' \in \bG_{k-n-1}} {(-1)^{\sigma' \cdot \tau}\left[ \F_k(\sigma',1,\sigma'') -\F_k(\sigma',0,\sigma'') \right]}\right| \\
  \leq & ~ 2^{-k}\sum_{\sigma'' \in \bG_{k-n-1}} \sum_{\sigma' \in \bG_n}{\left[ \F_k(\sigma',1,\sigma'') -\F_k(\sigma',0,\sigma'') \right]} \\
  \leq & ~2^{n-1}, 
 \end{align*}
since 
\begin{align*}
 \sum_{\sigma' \in \bG_n,}{\left[ \F_k(\sigma',1,\sigma'') -\F_k(\sigma',0,\sigma'') \right]}   \leq &  \sum_{s \in \{0,\ldots, 2^k-1\}}{\left[ \hF_k(s+1) -\hF_k(s) \right]} \\
 \leq & \sum_{s\in \{0,\ldots,2^k-1\} }{\frac{1}{\hh_k(s)\hh_k(s+1)}} \\
 \leq & 1
\end{align*}
for all $k\in \bN$ and $\sigma ''\in \bG_{k-n-1}$ due to equation \eqref{ggT1} and \ref{lem:sum}. 
\end{proof}

\subsection{Proof of Theorem \ref{theo:ferro}}
Now we present the proof of Theorem \ref{theo:ferro}. As we have seen before, the cone $\mC_k$ of strictly ferromagnetic observables ensures useful structural properties that are not fulfilled by weakly ferromagnetic observables. Therefore we transform our problem of weak ferromagnetism to one of strong ferromagnetism. 

First of all, we need two well known facts about the numerator and denominator functions.

\begin{lemma}[\cite{Knauf93}, Lem. 4.4]\label{lem:mul}
 For $k\in \bN$ and $s_0, s_1 \in \bR$ we have
\begin{align*}
 \zr_k(s_0,s_1)=s_0 \cdot \zr_k(1,0)+s_1 \cdot \zr_k(0,1). 
\end{align*}
\end{lemma}

\begin{lemma}[\cite{Knauf93}, Lem. 4.6]\label{lem:rek}
 For $k\in \bN$, $\sigma \in \bG_k$, $l\in \bN_0$, $\tau \in \bG_l$ and $s_0,s_1\in \bR$ we have 
\begin{align*}
 \zr_{k+l}(s_0,s_1)(\sigma, \tau) = \zr_{l+1}\left( \zr_k(s_0,s_1)(\sigma),\zr_k(s_0,s_1)(1-\sigma) \right)(0,\tau). 
\end{align*}
\end{lemma}

Using Lemma \ref{lem:mul} we get 
\begin{align*}
 \F_k(\sigma)=\frac{\zr_k(\sigma)}{\h_k(\sigma)} = \frac{\zr_{k+1}(0,1)(0,\sigma)}{\zr_{k+1}(1,1)(0,\sigma)} = 
\frac{1}{2}- \frac{1}{2}\,\frac{\zr_{k+1}(1,-1)(0,\sigma)}{\zr_{k+1}(1,1)(0,\sigma)} 
\end{align*}
for $\sigma \in \bG_k$, $k \in \bN$.  Therefore the interaction coefficients of $\F_k$ have the form 
\begin{align*}
 j_k(\tau) & = -2^{-k} \sum_{\sigma \in \bG_k}{\F_k(\sigma)(-1)^{\sigma \cdot \tau}} \\
& =  - 2^{-k-1} \sum_{\sigma \in \bG_k}{(-1)^{\sigma \cdot \tau}}  + 
2^{-k-1} \sum_{\sigma \in \bG_k}{\frac{\zr_{k+1}(1,-1)(0,\sigma)}{\zr_{k+1}(1,1)(0,\sigma)}(-1)^{\sigma \cdot \tau}} \\
& = -\frac{1}{2}\delta_{\tau,0} + \frac{1}{2} (\mF_k W_k)(\tau), 
\end{align*}
for all $\tau \in \bG_k$, in which $W_k:= \frac{\zr_{k+1}(1,-1)(0,\cdot)}{\zr_{k+1}(1,1)(0,\cdot)}$. 
We are going to prove that $W_k$  is a strictly ferromagnetic observable. This is a sufficient condition for the weak ferromagnetism of $-\F_k$ that was claimed in Theorem \ref{theo:ferro}.

\begin{lemma}\label{lem:ferrosym}
 For $k\in \bN$ we have $W_k \in \mC_k$. 
\end{lemma}

\begin{proof}
For $k=1$ we get
\begin{align*}
\left( \mF_1 W_1 \right)(\tau)  = \mF_1 \left(\frac{\zr_{2}(1,-1)(0,\cdot)}{\zr_{2}(1,1)(0,\cdot)} \right)(\tau)= \frac{1}{2}\left[1+(-1)^{\tau} \cdot 0 \right]= 
\frac{1}{2} \geq 0 
\end{align*}
for all $\tau\in \bG_1$. \\

For $i=1,2$ we study the M\"obius transformations $g_i\colon\,]-3,3[ \, \rightarrow \bR$ defined by 
\begin{align*}
g_1(x):=\frac{x+1}{-x+3} \text{~~~and~~~} g_2(x):=\frac{x-1}{x+3}. 
\end{align*}
Using \ref{lem:mul} we formally get
\begin{align*}
\frac{\zr_k(1,0)}{\zr_k(1,2)} = g_1\left(\frac{\zr_k(1,-1)}{\zr_k(1,1)}  \right) \text{~~~and~~~} 
\frac{\zr_k(0,-1)}{\zr_k(2,1)} = g_2\left(\frac{\zr_k(1,-1)}{\zr_k(1,1)} \right).  
\end{align*}
The latter compositions are well-defined due to Lemma \ref{lem:mul}. 
Furthermore the maps 
\begin{align*}
g_\pm\colon\,]-3,3[\,\rightarrow \bR, ~~~ g_\pm:=g_1 \pm g_2,
\end{align*}
are  real-analytic for all $|x|<3$ therefore their expansions are absolutely convergent on the interval $]-3,3[$. Additionally  
\begin{align*}
 \frac{d^n}{dx^n}g_\pm(x)|_{x=0} \geq 0 
\end{align*}
for all $n\in \bN_0$ because we have the following lemma:  
\begin{lemma}
For
\begin{align*}
\bR_{\geq 0}[x]:=\left\{\sum_{i \in \bN_0}{a_i x^i}: a_i \in \bR_{\geq 0} \text{~and~} a_i=0 \text{~for almost all~} i\in \bN_0  \right\}
\end{align*}
and $n \in \bN_0$ we have 
\begin{align*}
\frac{d^n}{dx^n}g_{\pm}(x)= (-1)^{n+1}\frac{p_\pm(x)}{(x^2-9)^{n+1}}          
\end{align*}
for all $x \in \,]-3,3[$ and a polynomial $p_\pm\in \bR_{\geq 0}[x]$ with $\operatorname{deg}(p_\pm)\leq n+2$.
\end{lemma}

\begin{proof}
For $n=0$ we have
\begin{align*}
 g_{+}(x)=-\frac{8x}{x^2-9}\text{~~~~~and~~~~~} g_{-}(x)=-\frac{2x^2+6}{x^2-9}.  
\end{align*}
By induction hypothesis we get $\frac{d^n}{dx^n}g_{\pm}(x)= (-1)^{n+1}\frac{p_\pm(x)}{(x^2-9)^{n+1}}$ with $p_\pm(x):=\sum_{i=0}^{N_\pm}{a_i^{\pm} x^i} \in \bR_{\geq0}[x]$ and $N_{\pm}:=\operatorname{deg}(p_\pm)\leq n+2 $. Therefore we have
\begin{align*}
 \frac{d^{n+1}}{dx^{n+1}}g_{\pm}(x) & = \frac{d}{dx}(-1)^{n+1}\frac{p_\pm(x)}{(x^2-9)^{n+1}} \\
				    & = (-1)^{n+2}\frac{9p_\pm'(x) - x^2 p_\pm'(x)+2(n+1)xp_\pm(x)}{(x^2-9)^{n+2}} \\
				    & = (-1)^{n+2}\frac{q_\pm(x)}{(x^2-9)^{n+2}},  
\end{align*}
in which $q_\pm(x):= 9p_\pm'(x) - x^2 p_\pm'(x)+2(n+1)xp_\pm(x)$.
Apparently  $\operatorname{deg}(q_\pm)\leq \operatorname{deg}(p_\pm)+1 \leq n+3$. A short calculation shows that 
\begin{align*}
q_\pm(x)  = & ~9 a_1^{\pm}+ \sum_{i=1}^{N_\pm-1}{ \left[ 9(i+1)a_{i+1}^{\pm}+(2(n+1)-(i-1))a_{i-1}^{\pm}\right] x^i} \\
& + \sum_{i=N_\pm}^{N_\pm+1}{ \left[2(n+1)-(i-1)\right]a_{i-1}^{\pm} x^i}.
\end{align*}
Due to $N_\pm\leq n+2$ for all $i\in \{1, \ldots N_\pm+1 \}$ we get $2(n+1)-(i-1) \geq n > 0$.
Therefore all coefficients of $q_\pm$ are non-negative.
\end{proof}
Since all  preconditions of Proposition \ref{pro:perma} are satisfied, we get by induction hypothesis that    
\begin{align*}
\frac{\zr_{k+1}(1,0)(0,\cdot)}{\zr_{k+1}(1,2)(0,\cdot)} \pm \frac{\zr_{k+1}(0,-1)(0,\cdot)}{\zr_{k+1}(2,1)(0,\cdot)} = \left( g_1 \pm g_2 \right)\left(\frac{\zr_{k+1}(1,-1)(0,\cdot)}{\zr_{k+1}(1,1)(0,\cdot)}  \right) 
\end{align*}
is strictly ferromagnetic. Therefore for $\tau:=(\tau_1, \ldots, \tau_{k+1})\in \bG_{k+1}$ and $\tau':=(\tau_2, \ldots, \tau_{k+1})$ we get by using Lemma \ref{lem:rek}
\begin{align*}
  &  \mF_{k+1} \left(\frac{\zr_{k+2}(1,-1)(0,\cdot)}{\zr_{k+2}(1,1)(0,\cdot)} \right)(\tau)  \\
= &\, 2^{-(k+1)} \sum_{\sigma \in  \bG_{k+1}}{\frac{\zr_{k+2}(1,-1)(0,\sigma)}{\zr_{k+2}(1,1)(0,\sigma)}(-1)^{\sigma \cdot \tau}} \\
= &\, 2^{-(k+1)}  \sum_{\sigma' \in  \bG_{k}}{\left[ \frac{\zr_{k+2}(1,-1)(0,0,\sigma')}{\zr_{k+2}(1,1)(0,0,\sigma')}
 +(-1)^{\tau_1}\frac{\zr_{k+2}(1,-1)(0,1,\sigma')}{\zr_{k+2}(1,1)(0,1,\sigma')} \right] (-1)^{\sigma' \cdot \tau'}} \\
= & \, 2^{-(k+1)}  \sum_{\sigma' \in  \bG_{k}}{\left[ \frac{\zr_{k+1}(\zr_2(1,-1)(0,0),\zr_2(1,-1)(1,1) )(0,\sigma')}
{\zr_{k+1}(\zr_2(1,1)(0,0),\zr_2(1,1)(1,1))(0,\sigma')} \right.} \\ 
 & \, {\left.+(-1)^{\tau_1}\frac{\zr_{k+1}(\zr_2(1,-1)(0,1),\zr_2(1,-1)(1,0))(0,\sigma')}
{\zr_{k+1}(\zr_2(1,1)(0,1),\zr_2(1,1)(1,0) )(0,\sigma')} \right] (-1)^{\sigma' \cdot \tau'}} \\
= & \, 2^{-(k+1)}  \sum_{\sigma' \in  \bG_{k}}{\left[ \frac{\zr_{k+1}(1,0)(0,\sigma')}{\zr_{k+1}(1,2)(0,\sigma')}
 +(-1)^{\tau_1}\frac{\zr_{k+1}(0,-1)(1,\sigma')}{\zr_{k+1}(2,1)(0,\sigma')} \right] (-1)^{\sigma' \cdot \tau'}} \\
= &\, 2^{-(k+1)}  \sum_{\sigma' \in  \bG_{k}}{\left(g_1+(-1)^{\tau_1}g_2\right) \left( \frac{\zr_{k+1}(1,-1)(0,\sigma')} 
{\zr_{k+1}(1,1)(0,\sigma')}\right)(-1)^{\sigma' \cdot \tau'}} \\
\geq &\, 0.  
\end{align*}
This completes the proof of Lemma \ref{lem:ferrosym}.
\end{proof}

\bigskip

\subsection{Proof of Theorem \ref{theo:inter}}

First of all we prove that \eqref{eq:inter} is summable for $\RE(s)>2$. Let $t\in [0,1]$ and $\sigma:=\RE(s)\geq 2+\delta$ for $\delta >0$. Then we have
\begin{align*}
 \sum_{\tau \in \bG_{\infty}}{\left| \frac{\exp\left(2 \pi i t (1-\F(\tau)) \right)}{h(\tau)^s} \right|}
 & =   \sum_{\tau \in \bG_{\infty}}{ h(\tau)^{-\sigma} }  = \sum_{n=1}^{\infty}{\varphi(n)n^{-\sigma}} \\
 & \leq \sum_{n=1}^{\infty}{n^{-1-\delta}} \leq  \sum_{n=1}^{\infty}{n^{-1-\delta}} \leq 1 +\frac{1}{\delta}. 
\end{align*}
Equation \eqref{eq:zustand} follows from \cite[Coro.\,2.3]{Knauf93}. Now we have to prove equation \eqref{eq:key}. The M\"obius function is denote by $\mu$. Using the well known facts that $1/\zeta(s)=\sum_{n=1}^\infty\mu(n)n^{-s}$ and $\mu(n)=\sum_{1\leq k\leq n, \gcd(k,n)=1}\exp(2\pi i k/n)$ (see e.\,g. \cite{Hardy60}) we get for $\sigma:=\RE(s)\geq 2+\delta$, $\delta>0$ 
 \begin{align*}
   &  \left| \sum_{n=1}^{\infty}{\mu(n)n^{-s}} - \sum_{\tau \in \bG_k}{\exp(-2 \pi i \F_k(\tau))(\h_k(\tau))^{-s}} \right| \\
 \leq & \sum_{n=k+2}^{\infty}{\left(\varphi(n)+\varphi_k(n) \right) n^{-\sigma}} \\
 \leq & 2\sum_{n=k+2}^{\infty}{ n^{-1-\delta}} 
\end{align*}
for all $k\in \bN$ using $\varphi_k\leq \varphi$ and $\varphi_k(p)=\varphi(p)$ for $1\leq p\leq k+1$. Taking the limit $k\to \infty$ the proof of Theorem \ref{theo:inter} is complete.

\bigskip
\bigskip

\paragraph{Acknowledgement}
I am very grateful to Andreas Knauf (Erlangen) for supporting this work.

\end{document}